\documentclass[reqno]{amsart}
\usepackage{url}
\usepackage{amsthm}
\usepackage{amsmath}
\usepackage{graphicx}
\usepackage{amssymb}

\newtheorem{theorem}{Theorem}[section]
\numberwithin{equation}{section}
\newcommand{\abs}[1]{\left\lvert#1\right\rvert}
\newcommand{\rd}[1]{\left\lfloor#1\right\rceil}
\newcommand{\ceiling}[1]{\left\lceil#1\right\rceil}
\newcommand{\gfileextn}{pdf}

\begin{document}

\title{A Tale of Two Omegas}

\author[M.~J. Mossinghoff]{Michael J. Mossinghoff}\thanks{This work was supported in part by a grant from the Simons Foundation (\#426694 to M.~J. Mossinghoff).}
\address{Department of Mathematics and Computer Science\\
Davidson College\\
Davidson, NC, 28035-6996, USA}
\curraddr{Center for Communications Research, 805 Bunn Dr., Princeton, NJ 08540 USA}
\email{m.mossinghoff@idaccr.org}

\author[T.~S. Trudgian]{Timothy S. Trudgian}\thanks{This work was supported by a Future Fellowship (FT160100094 to T.~S. Trudgian) from the Australian Research Council.}
\address{School of Science\\
UNSW Canberra at ADFA\\
ACT 2610, Australia}
\email{t.trudgian@adfa.edu.au}

\date{\today}

\begin{abstract}
\noindent
We consider $\omega(n)$ and $\Omega(n)$, which respectively count the number of distinct and total prime factors of $n$.
We survey a number of similarities and differences between these two functions, and study the summatory functions $L(x)=\sum_{n\leq x} (-1)^{\Omega(n)}$ and $H(x)=\sum_{n\leq x} (-1)^{\omega(n)}$ in particular.
Questions about oscillations in both of these functions are connected to the Riemann hypothesis and other questions concerning the Riemann zeta function.
We show that even though $\omega(n)$ and $\Omega(n)$ have the same parity approximately 73.5\% of the time, these summatory functions exhibit quite different behaviors: $L(x)$ is biased toward negative values, while $H(x)$ is unbiased.
We also prove that $H(x)>1.7\sqrt{x}$ for infinitely many integers $x$, and $H(x)<-1.7\sqrt{x}$ infinitely often as well.
These statements complement results on oscillations for $L(x)$.
\end{abstract}

\maketitle

\section{Introduction}\label{Intro}
\noindent
The two prime-counting functions $\omega(n)$ and $\Omega(n)$ arise frequently and naturally in many problems in number theory: recall that $\omega(n)$ designates the number of distinct prime factors of $n$, while $\Omega(n)$ denotes the total number of primes dividing $n$, counting multiplicity.
Thus, if $n=\prod_{i=1}^r p_i^{e_i}$ with $p_1$, \ldots, $p_r$ distinct primes and each $e_i$ positive, then $\omega(n)=r$ and $\Omega(n) = e_1+\cdots+e_r$.
Results concerning these closely related functions often occur in pairs in the literature and very often a result for one function closely resembles the corresponding result for the other.
However, significant differences arise as well: in this paper we investigate one area where the behavior attached to $\omega(n)$ is strikingly different from that of $\Omega(n)$.

For this, let $\lambda(n)$ denote the Liouville function, defined by $\lambda(n) = (-1)^{\Omega(n)}$, and let $\xi(n)$ denote the corresponding function for $\omega(n)$: $\xi(n) = (-1)^{\omega(n)}$.
Thus, $\lambda(n)$ indicates whether $n$ has an even or odd number of prime factors counting multiplicity, and $\xi(n)$ provides a similar indicator for the number of distinct prime factors of $n$.
Then define $L(x)$ and $H(x)$ as sums of these respective functions over the integers in $[1, x]$:
\[
L(x) = \sum_{n=1}^{\lfloor x\rfloor} \lambda(n), \quad H(x) = \sum_{n=1}^{\lfloor x\rfloor} \xi(n).
\]
P\'olya studied values of $L(x)$ --- which here we will call P\'{o}lya's function --- in connection with a question on quadratic forms.
In 1919 \cite{Polya}, he showed that if $p$ is a prime with $p\equiv3$ mod 4, $p>7$, and the number of classes of positive definite quadratic forms with discriminant $-p$ is $1$, then $L((p-3)/4) = 0$.
This therefore holds for $p=11$, $19$, $43$, $67$, and $163$ (the qualifying primes where $\mathbb{Q}(\sqrt{-p})$ has class number $1$).
In this article, P\'olya reported computing $L(n)$ for $n$ up to about 1500, and he noted that this function was never positive over this range, after $L(1)=1$.
He remarked that this merited further study, since if it could be shown that $L(x)$ never changed sign for sufficiently large $x$, then the Riemann hypothesis would follow, as well as the statement that all of the zeros of the Riemann zeta function are simple.
The statement that $L(x)<0$ for all $x>1$, or more conservatively for all sufficiently large $x$, is sometimes referred to as \textit{P\'olya's conjecture} in the literature, but for historical accuracy we remark that P\'olya did not in fact conjecture this inequality, at least not in print, and so we prefer the term \textit{P\'olya's question} for this problem.

P\'olya's question is reminiscent of an older problem regarding the sum of the values of another common arithmetic function in number theory, the M\"obius function $\mu(n)$, a close cousin of the Liouville function.
(Recall that $\mu(n)=\lambda(n)$ if $\Omega(n)=\omega(n)$, and 0 otherwise.)
Define the \textit{Mertens function} by
\[
M(x) = \sum_{n=1}^{\lfloor x\rfloor} \mu(n).
\]
In 1897 Mertens \cite{Mertens} tabulated values of this function for $n\leq10^4$, and on the basis of those results he opined that it was ``very probable'' that $M(x) < \sqrt{x}$ for all $x>1$.
This is widely known as the \textit{Mertens conjecture}, and this problem remained open until 1985, when it was disproved by Odlyzko and te Riele \cite{OTR}.
It is interesting to note that the behavior of the Mertens function is qualitatively quite different from that of P\'olya's function: the latter is skewed toward negative values, while the former oscillates more evenly around $0$. 
Figure~\ref{figMLH} displays plots of both of these functions, normalized by dividing by $\sqrt{x}$, over the interval $[e^9, e^{14}]$.
Here we replace $x$ with $e^u$ and plot in $u$ in order to display the oscillations  more effectively.
We note that this qualitative difference appears despite the fact that the M\"obius and Liouville functions agree on a substantial subset of the positive integers. Indeed, $\lambda(n)=\mu(n)$ precisely when $n$ is squarefree, and it is well known that the set of squarefree numbers has density $6/\pi^2=0.6079\ldots$ in the set of positive integers.

It is well known that the Riemann hypothesis and the simplicity of the zeros of the zeta function would both follow if it could be established that either of the functions $L(x)/\sqrt{x}$ or $M(x)/\sqrt{x}$ were bounded, either above or below (see for instance \cite{MosTru}).
However, in 1942 Ingham \cite{Ingham1942} showed that bounding either of these functions, in either direction, implies significantly more: it would also follow that there exist infinitely many integer linear relations among the ordinates of the zeros of the Riemann zeta function on the critical line in the upper half-plane.
Stated another way, if one assumes the Riemann hypothesis and the simplicity of the zeros of the zeta function, and if the ordinates of its zeros on the critical line in the upper half plane are linearly independent over $\mathbb{Z}$, then both $L(x)/\sqrt{x}$ and $M(x)/\sqrt{x}$ must have unbounded\footnote{We shall return to these oscillations in Section~\ref{secOpen}.} limit supremum and limit infimum.
Since there seems to be no reason to suspect that such linear relations exist, and certainly none have been detected computationally, it seems plausible that the oscillations in these functions are in fact unbounded.

Beginning in the late 1960s, Bateman et al.\ \cite{Bateman}, Grosswald \cite{Grosswald67, Grosswald72}, Diamond \cite{Diamond}
and others developed a method for obtaining information on the oscillations in functions like $M(x)$ and $L(x)$.
The main idea is that if a subset of the nontrivial zeros of the zeta function could be shown to have imaginary parts that are \textit{weakly} independent, then it would follow that the functions in question must have \textit{large} oscillations.
This method is described in Section~\ref{secIndependence}: roughly speaking, we define a set to be $N$-independent over $\mathbb{Z}$ if no nontrivial integer relations exist when the coefficients are bounded by $N$ in absolute value.
An effective version of this method was employed by Best and the second author \cite{Best} in 2015 to establish that
\begin{equation}\label{eqnOscMerBT}
\liminf_{x\to\infty} \frac{M(x)}{\sqrt{x}} < -1.6383, \quad
\limsup_{x\to\infty} \frac{M(x)}{\sqrt{x}} > 1.6383.
\end{equation}
(These values have since improved by Hurst \cite{Hurst} in 2016, who employed a method based in part on the work of Odlyzko and te Riele to obtain $-1.8376$ in the first inequality and $1.8260$ in the second.)
In 2018, the authors \cite{MT2017} employed this weak independence method to obtain new results on oscillations for $L(x)$:
\begin{equation}\label{eqnOscPolMT}
\liminf_{x\to\infty} \frac{L(x)}{\sqrt{x}} < -2.3723, \quad
\limsup_{x\to\infty} \frac{L(x)}{\sqrt{x}} > 1.0028.
\end{equation}

The function $H(x)$ and the associated Dirichlet series $h(s) = \sum_{n\geq 1} \xi(n) n^{-s}$, were investigated by van de Lune and Dressler \cite{Dressler} in 1975.
They proved that $H(x) = o(x)$ and that $h(1) = 0$. This is what one expects on analogy with sums of $\lambda(n)$. Indeed, bounds on $L(x)$ and $H(x)$ can be derived, via Dirichlet convolution identities, from bounds on $M(x)$, which themselves follow directly from the Prime Number Theorem. For example, we have $\xi = \mu \ast \eta$, where $\eta(n)$ is the multiplicative function defined on prime powers as $\eta(p^{k}) = 1- k$. Thus Schwarz \cite{Schwarz} noted that $H(x) = O(x \exp(-c\sqrt{\log x}))$;
this may be improved further by using the Vinogradov--Korobov zero-free region of the Riemann zeta function to $H(x) = O(x \exp(-c'(\log x)^{3/5} (\log\log x)^{-1/5}))$.

In this article, we investigate the oscillations in $H(x)$, similar to the studies on $M(x)$ and $L(x)$.
The behavior of $H(x)/\sqrt{x}$ is linked to the Riemann hypothesis and the linear independence question in the same way, but the qualitative behavior of this function is different.
Figure~\ref{figMLH} also displays the normalized values $H(x)/\sqrt{x}$ over $[e^9, e^{14}]$.
For a wider view, Figure~\ref{figBigH} exhibits $H(x)/\sqrt{x}$ over $x\leq 1.5\cdot10^{15}$.
(The plots for the other two functions continue roughly in the same horizontal bands evident in Figure~\ref{figMLH} and so are not displayed on this wider scale.)
This striking qualitative difference appears despite the fact that $\lambda(n)$ and $\xi(n)$ agree on a set of significantly higher density than that of the set of squarefree numbers: we prove the following theorem in Section~\ref{secParity}.

\begin{theorem}\label{thmAgree}
Let $\beta(x)$ denote the number of positive integers $n\leq x$ for which $\omega(n) \equiv \Omega(n) \pmod{2}$. Then $\beta = \lim_{x\rightarrow\infty} \beta(x)/x$ exists and $\beta = 0.73584\ldots.$
\end{theorem}

\begin{figure}[tb]
\centering
\includegraphics[width=4in]{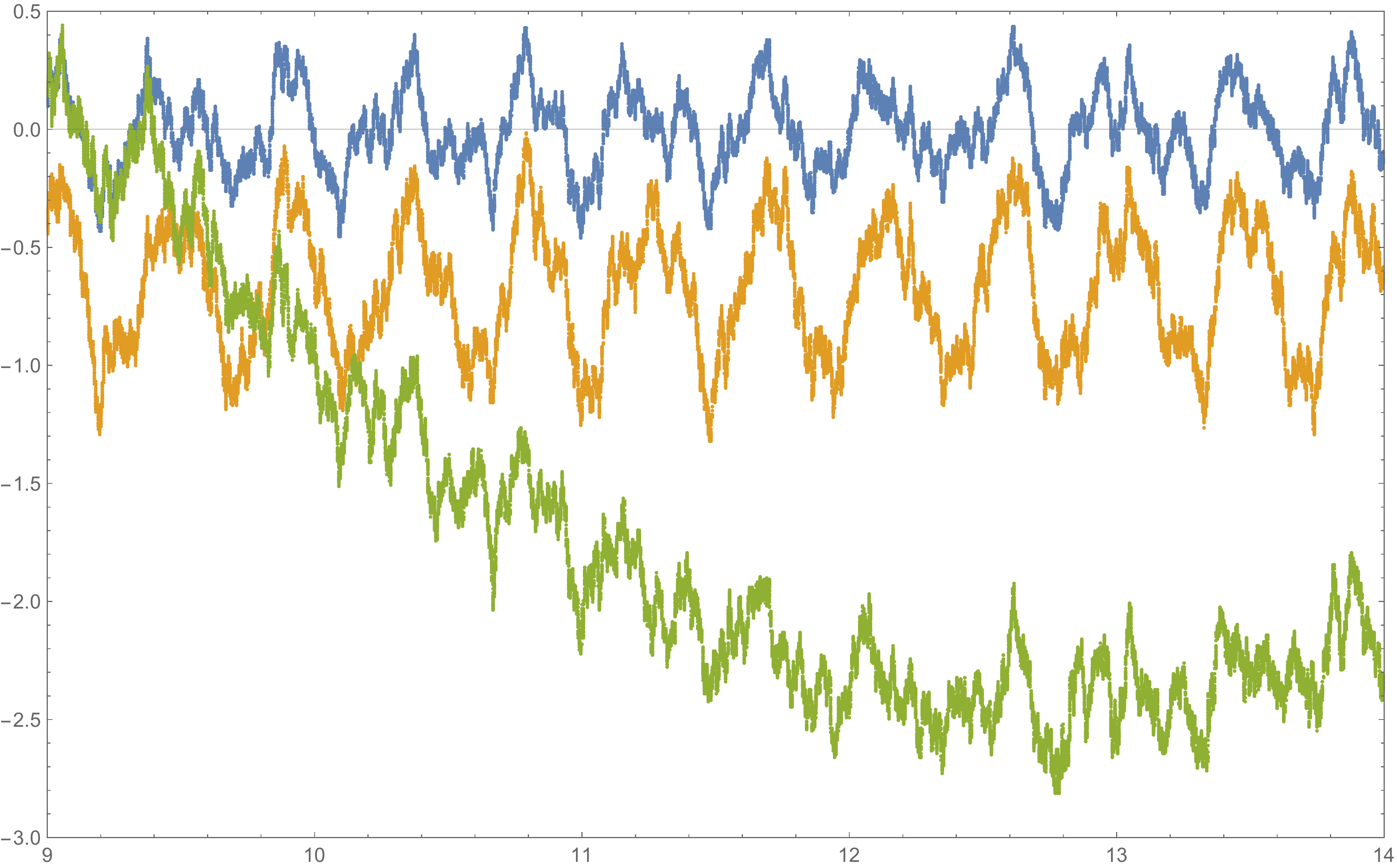}
\caption{$e^{-u/2}M(e^u)$ (top, oscillating about $y=0$), $e^{-u/2}L(e^u)$ (middle, oscillating about $y=1/\zeta(1/2)=-0.6847\ldots$), and $e^{-u/2}H(e^u)$ (bottom), for $9\leq u\leq 14$.}\label{figMLH}
\end{figure}

\begin{figure}[tb]
\centering
\includegraphics[width=4in]{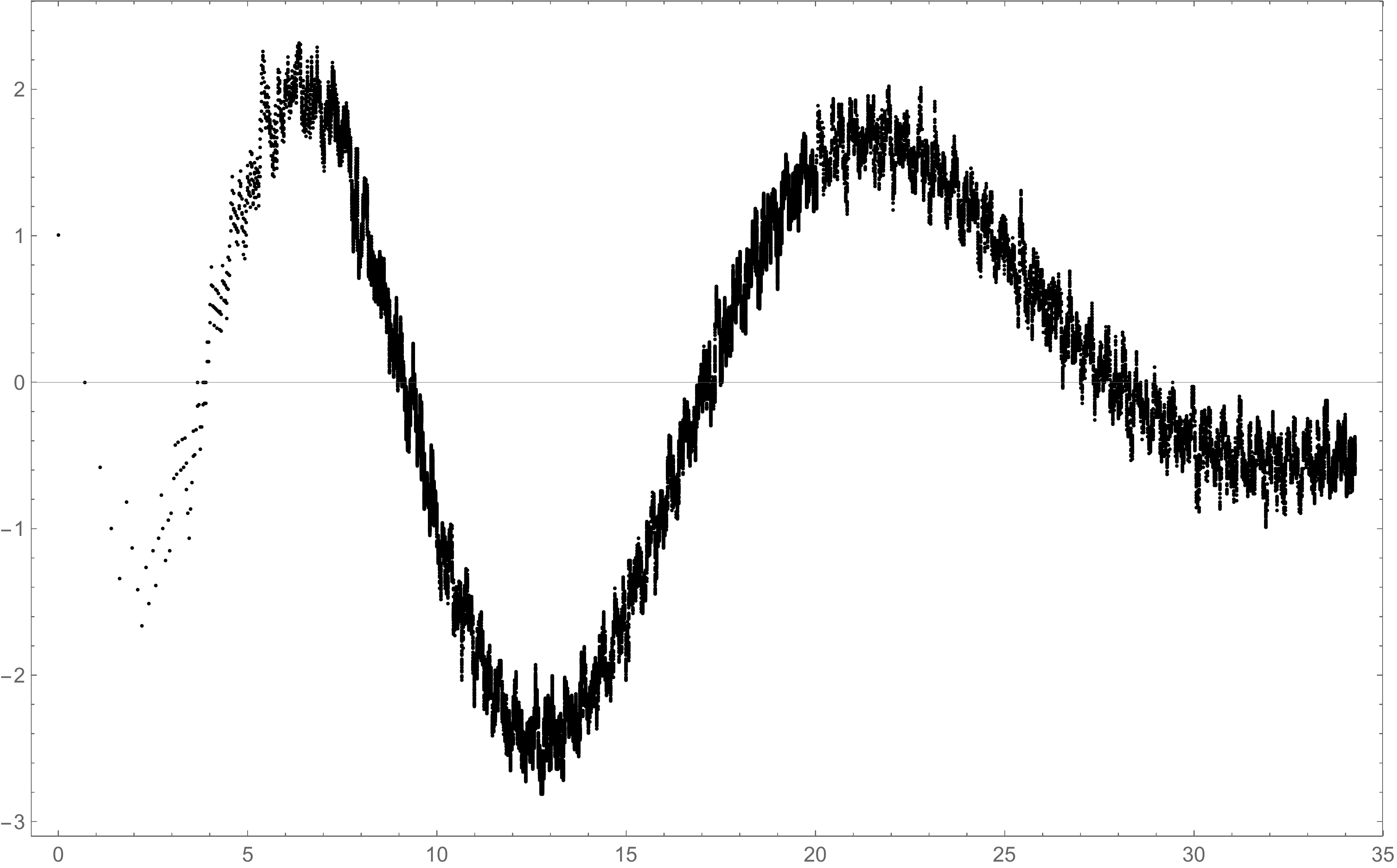}
\caption{$e^{-u/2}H(e^u)$ for $0\leq u\leq \log(1.5\cdot10^{15})$.}\label{figBigH}
\end{figure}

We also establish lower bounds on the oscillations in $H(x)$ by establishing weak independence for the ordinates of certain zeros of the Riemann zeta function, following our work in \cite{MT2017}.
We prove the following theorem in Section~\ref{secCalculations}.

\begin{theorem}\label{thmOscillate}
We have
\begin{equation*}\label{humble}
\liminf_{x\to\infty} \frac{H(x)}{\sqrt{x}} < -1.700144 \quad\textrm{and}\quad
\limsup_{x\to\infty} \frac{H(x)}{\sqrt{x}} > 1.700144.
\end{equation*}
\end{theorem}

In particular, we show that under the assumption of the Riemann hypothesis and the simplicity of the zeros of $\zeta(s)$, the function $H(x)$ does not appear to be biased toward positive or negative values, quite unlike $L(x)$, but like the Mertens function. This is perhaps surprising given the initial bias recorded in Figure \ref{figMLH}, though, as seen in Figure \ref{figBigH}, this initial bias eventually dissipates. 

This article is organized in the following way.
Section~\ref{secSurvey} surveys results in the literature concerning the two functions $\omega(n)$ and $\Omega(n)$, highlighting a number of rather similar behaviors.
Section~\ref{secParity} investigates the parity of these two functions and proves that the density of positive integers where $\omega(n) \equiv \Omega(n)$ (mod $2$) is $0.73584\ldots$\,, establishing Theorem~\ref{thmAgree}.
Section~\ref{secIndependence} illustrates how $L(x)$ and $M(x)$ are connected to the Riemann hypothesis and other questions concerning the Riemann zeta function, and describes Grosswald's notion of weak independence and how it can be used to establish large oscillations in these summatory functions.
Section~\ref{secDirichletomega} analyzes the function $H(x)$ and an associated Dirichlet series in preparation for applying the machinery of weak independence to the question of oscillations in this function.
Section~\ref{secQualAn} then provides a qualitative analysis of $L(x)$ and $H(x)$ and their associated approximating series.
Section~\ref{secCalculations} summarizes our calculations on this problem, establishes bounds on the oscillations in $H(x)$, and proves Theorem~\ref{thmOscillate}.
It also describes our method for calculating values of $\xi(n)$ over an interval, which was used to generate the data for Figure~\ref{figBigH}.
Finally, Section~\ref{secOpen} lists some open problems regarding parity issues and related phenomena for $\Omega(n)$ and $\omega(n)$.

\section{Some results on $\omega(n)$ and $\Omega(n)$}\label{secSurvey}

We survey a number of results concerning $\Omega(n)$ and $\omega(n)$.
While this survey is not exhaustive, we include a number of results that highlight the rather similar behavior of $\omega(n)$ and $\Omega(n)$ in a number of analytic problems involving sums, order computations, distributional properties, and the attainment of special sequences of values.
For additional examples, we direct the reader to the reference book by S\'andor, Mitrinovi\'c, and Crstici \cite{Sandy}.

\subsection{Sums involving $\omega(n)$ and $\Omega(n)$}\label{subsecSumsOmegas}

Since we are interested in the partial sums $\sum_{n\leq x} (-1)^{\omega(n)}$ and $\sum_{n\leq x} (-1)^{\Omega(n)}$, the average orders $\sum_{n\leq x} \omega(n)$ and $\sum_{n\leq x} \Omega(n)$ are natural to consider at a first pass.
Hardy and Ramanujan \cite{HRam} proved that
\begin{equation}\label{throne}
\begin{split}
\sum_{n\leq x} \omega(n) &= x \log\log x + Ax + O\left(\frac{x}{\log x}\right),\\
\sum_{n\leq x} \Omega(n) &= x \log\log x + Bx + O\left(\frac{x}{\log x}\right),
\end{split}
\end{equation}
where
\[
A = \gamma + \sum_p \left(\log\left(1-\frac{1}{p}\right) + \frac{1}{p}\right) = 0.26149\ldots,
\]
and
\[
B = A + \sum_p \frac{1}{p(p-1)} = 1.03465\ldots.
\]
In the same paper, Hardy and Ramanujan proved that these two functions have the same normal order, that is,
\[
\bigg|\left\{n\leq x:\quad |f(n) - \log\log n| > (\log\log n)^{1/2+ \delta}\right\}\bigg| = o(x),
\]
for any positive $\delta$, where $f(n)$ may be $\omega(n)$ or $\Omega(n)$.
In connection with this, we mention the Erd\H{o}s--Kac theorem, which states that $\omega(n)$ and $\Omega(n)$ are normally distributed.
More precisely, for any fixed $y$, as $x\rightarrow\infty$ we have
\begin{equation*}\label{line}
x^{-1} \left| \left\{n\leq x: \quad \frac{g(n) - \log\log n}{\sqrt{\log\log n}} \leq y\right\}\right| \rightarrow \frac{1}{\sqrt{2\pi}} \int_{-\infty}^{y} \exp\left(-\frac{w^{2}}{2}\right)\, dw,
\end{equation*}
where $g(n)$ may be $\omega(n)$ or $\Omega(n)$.

The fact that $\omega(n)$ and $\Omega(n)$ have the same average order, as given in \eqref{throne}, can be compared with a result by Duncan \cite{MacBeth}:
\begin{equation}\label{charles}
\sum_{2\leq n \leq x} \frac{\Omega(n)}{\omega(n)} = x + O(x/\log\log x),
\end{equation}
which\footnote{We note that the error term in \eqref{charles} has been improved by De Koninck \cite{Kon74} and by Ivi\'{c} and De Koninck \cite{Kon80}.} says that on average $\Omega(n) = \omega(n)$.
This statement is also expressed by Ivi\'{c} \cite{Ivic88}, who proved that for $0<\delta<1$ and for $\delta \leq a \leq 2- \delta$ we have 
\[
\sum_{2\leq n \leq x} a^{\Omega(n)/\omega(n)} = ax + O\left(x(\log\log x)^{\epsilon -1}\right).
\]

Similarities also abound in moments of inverse powers of the omegas.
For instance, De Koninck \cite{Kon72} showed that
\[
\sum_{2\leq n \leq x} \frac{1}{\omega(n)} \sim \frac{x}{\log\log x} \sim  \sum_{2\leq n \leq x} \frac{1}{\Omega(n)}
\]
and
\[
\sum_{2\leq n \leq x} \frac{1}{\omega(n)^{2}} \sim \frac{x}{(\log\log x)^{2}} \sim  \sum_{2\leq n \leq x} \frac{1}{\Omega(n)^{2}},
\]
where each time the lower order terms are slightly different.
Some differences arise in other sums involving inverses: consider
\[
S_{1} = \sum_{2\leq n \leq x} n^{-1/\omega(n)}, \quad S_{2} = \sum_{2\leq n \leq x} n^{-1/\Omega(n)}.
\]
A result of Xuan \cite{Xuan} (improved in \cite{Xuan2}), which answers a question of Erd\H{o}s and Ivi\'{c} \cite{Ivic82}, states that there exist positive constants $c_1$ and $c_2$ so that
\[
S_{1} = x \exp\left(-(c_{1} + o(1))(\log x \log\log x)^{1/2}\right),
\]
but
\[
 S_{2} = x\exp\left(-(c_{2} + o(1)) \log^{1/2} x\right).
\]

\subsection{Sequences of special values}\label{subsecSeqSpecVals}

Here we collect some results on special values of these two omega functions.

Heath-Brown \cite{RHB} proved that $\Omega(n) = \Omega(n+1)$ for infinitely many $n$, and Schlage-Puchta \cite{Puchta} proved the analogous result for $\omega(n)$ by using the same method. Furthermore, Goldston, Graham, Pintz, and Y{\i}ld{\i}r{\i}m \cite{card} proved that $\omega(n) = \omega(n+1) = 3$ and $\Omega(n) = \Omega(n+1) = 4$ for infinitely many $n$.
This constancy in the omegas cannot occur too frequently, though, owing to the work of Erd\H{o}s, Pomerance, and S\'{a}rk\"{o}zy (see \cite{EPS2} and \cite{EPS3}), who showed that
\[
\#\{n\leq x \;:\; \omega(n) = \omega(n+1)\} = O\left(x/\sqrt{\log\log x}\right).
\]

Cooper and Kennedy \cite{Cooper} showed that it is unusual for $\omega(n)$ to divide $n$:
\begin{equation}\label{william}
\lim_{x\rightarrow\infty} \frac{\big|\{n\leq x \;:\; \omega(n)\mid n\}\big|}{x} = 0.
\end{equation}
Erd\H{o}s and Pomerance \cite{EPom} extended this by proving a more general theorem, which both recovers \eqref{william} and also establishes the analogous result for $\Omega(n)$.

One might also ask how often one of these omega functions produces an output that is relatively prime to its input.
Let
\[
A_{1}(x) = \#\{ n \leq x\; :\; (n, \Omega(n)) = 1\}
\]
and
\[
A_{2}(x) = \#\{ n \leq x\; :\; (n, \omega(n)) = 1\}.
\]
Vol'kovi\v{c} \cite{Volt} proved that $A_{1}(x) \sim 6x/\pi^2$, so this occurs with the same density as the squarefree numbers.
For $\Omega(n)$, Alladi \cite{Alladi} showed that $A_{2}(x) \sim 6x/\pi^2$ as well, and also produced an estimate of the error term.

\section{The parity of $\Omega(n) - \omega(n)$}\label{secParity}

As shown in Figure~\ref{figMLH}, the functions $L(x)$ and $H(x)$ appear to have qualitatively different behavior.
This may be surprising at first glance since the summands in these two functions are equal more often than not.
Certainly $\Omega(n) = \omega(n)$ if $n$ is squarefree, so for at least $6/\pi^2 = 0.6079\ldots$ of integers $n$ we have equal summands in $L(x)$ and $H(x)$.
Indeed, the summands are equal precisely when
\begin{equation}\label{duke}
\omega(n) \equiv \Omega(n) \pmod{2}.
\end{equation} 
The purpose of this section is to investigate the frequency for which \eqref{duke} holds in a more precise way, and establish the more exact estimate of Theorem~\ref{thmAgree}.

We note that a related result was first proved by R\'{e}nyi \cite{Ren} in 1955.
For $k\geq 0$, let $N_{k}(x)$ denote the number of integers not exceeding $x$ for which $\Omega(n) - \omega(n) = k$.
This is clearly reminiscent of (\ref{duke}).
R\'{e}nyi showed that $N_{k}(x) \sim d_{k} x$, where the $d_k$ are the power series coefficients of the meromorphic function
\begin{equation}\label{eqnRenyi}
\begin{split}
R(z) = \sum_{k\geq0} d_k z^k &= \prod_p\left(1-\frac{1}{p}\right)\left(1+\frac{1}{p-z}\right)\\
&= \frac{1}{\zeta(2)}\prod_p \left(1+\frac{z}{(p+1)(p-z)}\right),
\end{split}
\end{equation}
which converges on $\abs{z}<2$.
R\'enyi's work was later extended by Robinson \cite{Rob}.
For a detailed exposition of these results one may consult Section 2.14 and p.\ 71 in \cite{MV}.
It follows from \eqref{eqnRenyi} then that the limiting density $\beta$ of Theorem~\ref{thmAgree} exists and that $\beta=\sum_{k\geq0} d_{2k} = (1+R(-1))/2$.
We include an elementary and self-contained proof of Theorem~\ref{thmAgree} here, which also provides a means for estimating $\beta$.

\begin{proof}[Proof of Theorem~\ref{thmAgree}]
We first determine the density of positive integers $n$ whose powerful part $A$ has a fixed set of prime divisors $P=\{p_1,\ldots,p_r\}$.
That is, $n=AB$ with $A=p_1^{a_1+2}\cdots p_r^{a_r+2}$, each $a_i\geq0$, $B$ squarefree, and $\gcd(A,B)=1$.
The density of integers that are squarefree except possibly at the primes $p_1$, \ldots, $p_r$ is
\[
\prod_{p\not\in P} \left( 1 - \frac{1}{p^{2}}\right) = \frac{1}{\zeta(2)}\prod_{i=1}^r \frac{p_i^2}{p_i^2-1}.
\]
Fix nonnegative integers $a_1$, \ldots, $a_r$.
The density of integers with powerful part $A$ is then
\[
\frac{6}{\pi^2}\prod_{i=1}^r \frac{p_i^2}{p_i^2-1} \cdot \prod_{i=1}^r \left(\frac{1}{p_i^{a_i+2}} - \frac{1}{p_i^{a_i+3}}\right)
= \frac{6}{\pi^2}\prod_{i=1}^r \frac{1}{p_i^{a_i+1}(p_i+1)}.
\]
We next sum this over allowable values for the $a_i$: we require that $a_1+\cdots+a_r\equiv r$ mod $2$, with each $a_i\geq0$.
Let $\mathcal{A}$ denote the set of allowable integer vectors $\mathbf{a}=(a_1,\ldots,a_r)$.
Now
\begin{equation}\label{eqnMatch}
\begin{split}
\sum_{\mathbf{a}\in\mathcal{A}} \frac{6}{\pi^2}\prod_{i=1}^r \frac{1}{p_i^{a_i+1}(p_i+1)}
&=
\frac{6}{\pi^2}\prod_{i=1}^r \frac{1}{p_i(p_i+1)} \sum_{\mathbf{a}\in\mathcal{A}} \prod_{i=1}^r \frac{1}{p_i^{a_i}}\\
&=
\frac{6}{\pi^2}\prod_{i=1}^r \frac{1}{p_i(p_i+1)} \biggl(\sum_{\substack{\mathbf{v}\in\{0,1\}^r\\\sum_{i=1}^r v_i \equiv r (2)}} \prod_{i=1}^r \sum_{k\geq0} \frac{1}{p_i^{2k+v_i}}\biggr)\\
&=
\frac{6}{\pi^2}\prod_{i=1}^r \frac{1}{p_i(p_i+1)} \biggl(\sum_{\substack{\mathbf{v}\in\{0,1\}^r\\\sum_{i=1}^r v_i \equiv r (2)}} \prod_{i=1}^r \frac{p_i^{2-v_i}}{p_i^2-1}\biggr)\\
&=
\frac{6}{\pi^2}\prod_{i=1}^r \frac{1}{(p_i+1)(p_i^2-1)} \biggl(\sum_{\substack{\mathbf{v}\in\{0,1\}^r\\\sum_{i=1}^r v_i \equiv r (2)}} \prod_{i=1}^r p_i^{1-v_i}\biggr)\\
&=
\frac{6}{\pi^2}\prod_{i=1}^r \frac{1}{(p_i+1)(p_i^2-1)} \biggl(\sum_{\substack{\mathbf{v}\in\{0,1\}^r\\\sum_{i=1}^r v_i \equiv 0 (2)}}  \prod_{i=1}^r p_i^{v_i}\biggr).
\end{split}
\end{equation}
This is then the density of integers with multiplicity greater than $1$ on precisely the set of primes $P$, and which have $\Omega(n) \equiv \omega(n)$ (mod $2$).

In the same way, one can show that the density of integers with powerful part whose prime set is precisely $P$, but where $\Omega(n) \not\equiv \omega(n)$ (mod $2$), is
\begin{equation}\label{eqnNonmatch}
\frac{6}{\pi^2}\prod_{i=1}^r \frac{1}{(p_i+1)(p_i^2-1)} \biggl(\sum_{\substack{\mathbf{v}\in\{0,1\}^r\\\sum_{i=1}^r v_i \equiv 1 (2)}}  \prod_{i=1}^r p_i^{v_i}\biggr).
\end{equation}

We may now compute a lower estimate for the density of the integers where $\Omega(n) \equiv \omega(n)$ (mod $2$) by summing the values of \eqref{eqnMatch} over various sets of primes.
We may compute an upper estimate by performing the same procedure on \eqref{eqnNonmatch} and subtracting the result from $1$.
The existence of the limiting density $\beta$ then follows by showing that the upper and lower estimates converge to the same value.
Indeed, the difference in the upper and lower estimates for a fixed set $P=\{p_1,\ldots,p_r\}$ is
\[
1 - \frac{6}{\pi^2}\prod_{i=1}^r \frac{1}{(p_i+1)(p_i^2-1)} \biggl(\sum_{\mathbf{v}\in\{0,1\}^r}  \prod_{i=1}^r p_i^{v_i}\biggr)
= 1 - \frac{6}{\pi^2}\prod_{i=1}^r \frac{1}{p_i^2-1}.
\]
Summing the latter term over all finite subsets $P\subset\mathcal{P}$, where $\mathcal{P}$ denotes the set of all positive primes, we see that
\[
\frac{6}{\pi^2} \sum_P \prod_{p\in P} \frac{1}{p^2-1}
= \frac{6}{\pi^2} \sum_P \prod_{p\in P} \sum_{k\geq1} p^{-2k}
= \frac{6}{\pi^2} \prod_{p\in\mathcal{P}} \biggl(1+\sum_{k\geq1} p^{-2k}\biggr)
= \frac{6}{\pi^2}\zeta(2)
= 1,
\]
so it follows that the limiting density $\beta$ exists.

To estimate its value, we employ a number of subsets of primes.
For $r=1$, we use all primes up to $3\cdot 10^6$; for $r=2$, we consider all pairs of primes with each at most $17500$; for $r=3$, the bound is $1500$; for $r=4$, we use $450$; $r=5$, $250$; and $r=6$, $170$.
(Of course, $r=0$ provides a contribution of $6/\pi^2$ to the lower bound as well.)
This allows us to compute that
\[
0.735836 < \beta < 0.735844.
\]

We can improve these bounds by considering instead all sets of primes $P=\{p_1,\ldots,p_r\}$ for which the product $p_1\cdots p_r$ is bounded by a particular constant $B$.
Using $B=10^7$, we compute
\[
0.735840285 < \beta < 0.735840329.\qedhere
\]
\end{proof}

We conclude this section by noting that Detrey, Spaenlehauer, and Zimmermann \cite{Zim} recently reported the value of $\beta$ to 1000 decimal digits by using a derivation of $\beta$ as $(1+R(-1))/2$, with $R(z)$ given by \eqref{eqnRenyi}, and a computational strategy of Moree \cite{Moree}.
They find
\[
\beta = 0.735840306806498934037617816540241043712963\ldots.
\]

\section{Weak independence and oscillations}\label{secIndependence}

Classical methods (see, e.g., \cite[Chapter 15]{MV}) show that if either the Riemann hypothesis fails or if a nontrivial zero is not simple, then
\begin{equation*}\label{archie}
\limsup_{x\rightarrow\infty} \frac{L(x)}{x^{1/2}} = \infty, \quad\liminf_{x\rightarrow\infty} \frac{L(x)}{x^{1/2}} =-\infty,  
\end{equation*}
and the same holds for $M(x)/\sqrt{x}$ and $H(x)/\sqrt{x}$.
Therefore, to exhibit extreme values of these functions, we may as well assume both the Riemann hypothesis and the simplicity of the zeros. 
In this section, we briefly review Ingham's work connecting oscillations to the linear independence problem for the nontrivial zeros of the zeta function.
We then recall how establishing weak independence for certain subsets produces lower bounds on oscillations for such functions.
For additional details, see \cite{Ingham1942,MosTru,MT2017}.

Following Ingham, let $A(u)$ denote a real function on $[0,\infty)$ that is absolutely integrable over every finite interval $[0,U]$, with Laplace transform
\begin{equation}\label{banquo}
F(s) = \int_0^\infty A(u) e^{-su}\,du,
\end{equation}
which is convergent in a half-plane $\sigma > \sigma_1 \geq0$.
Suppose that $F(x)$ can be analytically continued to the region $\sigma\geq0$, except for a sequence of simple poles that occur symmetrically along the imaginary axis, with one possibly at the origin, and the others at locations $\{i\gamma_n,-i\gamma_n\}$ for an increasing sequence of positive real numbers $\{\gamma_n\}$.
Let $r_{\gamma_n}$ denote the residue of $F(s)$ at $i\gamma_n$, and let $r_0$ denote the residue at the origin, if indeed there is a pole there, otherwise take $r_0=0$.
Next, fix a real number $T>1$, and let $k_T(u)$ be an even function which is supported on $[-T,T]$, has $k_T(0)=1$, and is the Fourier transform of a nonnegative function $K_T(y)$.
Ingham employed the F\'ejer kernel,
\begin{equation}\label{eqnFejerkernel}
k_T(t) =
\begin{cases}
1 - \abs{t}/T, & \abs{t} \leq T,\\
0, & \abs{t} > T.
\end{cases}
\end{equation}
but, following \cite{Best,MT2017,OTR} we use the somewhat more advantageous kernel of Jurkat and Peyerimhoff \cite{JP},
\begin{equation}\label{eqnJPkernel}
k_T(t) =
\begin{cases}
(1-\abs{t}/T)\cos(\pi t/T) + \sin(\pi\abs{t}/T)/\pi, & \abs{t} \leq T,\\
0, & \abs{t} > T.
\end{cases}
\end{equation}
Next, let
\[
B_T^*(u) = r_0 + 2\Re\sum_{0<\gamma_n<T} k_T(\gamma_n) r_{\gamma_n} e^{i\gamma_n u}.
\]
Ingham's work then shows that for any positive real number $v$,
\[
\liminf_{u\to\infty} A(u)
\leq \liminf_{u\to\infty} B^*_T(u)
\leq B_T^*(v)
\leq \limsup_{u\to\infty} B^*_T(u)
\leq \limsup_{u\to\infty} A(u).
\]
Haselgrove \cite{Haselgrove} employed this approach to establish a negative answer to P\'olya's question: taking $T=1000$, he computed that $B_T^*(831.846)>0$ and $B_T^*(853.853)<0$ for the function $A_L(u) = e^{-u/2}L(e^u)$.
Note in this case the Laplace transform is
\begin{equation}\label{fleance}
F_L(s) = \frac{\zeta(2s+1)}{(s+1/2)\zeta(s+1/2)},
\end{equation}
which (on the Riemann hypothesis and the simplicity of the zeros of $\zeta(s)$) has simple poles at $s=0$ and at $s=i\gamma_n$, where the nontrivial zeros of the zeta function in the upper half plane are $\rho_n=1/2+i\gamma_n$ with $\{\gamma_n\}$ increasing.

When the poles of $F(s)$ on the positive imaginary axis are the zeros of the zeta function, Ingham then showed that as long as $\sum_{n\geq1} \abs{r_{\gamma_n}}$ diverges, and there are only finitely many integer relations among the values $\{\gamma_n\}$, then it follows that $\liminf _{u\to\infty}A(u) = -\infty$ and $\limsup_{u\to\infty} A(u) = \infty$.

Odlyzko and te Riele employed this result in their disproof of the Mertens conjecture.
Here,
\begin{equation}\label{escape}
F_M(s) = \frac{1}{(s+1/2)\zeta(s+1/2)},
\end{equation}
so there is no pole at the origin ($r_0=0$), and the sequence of residues is
\[
r_\gamma = \frac{1}{\rho\zeta'(\rho)},
\]
where we write $\rho$ for $1/2+i\gamma$.
They selected the $70$ largest residues $r_\gamma$ among the first $1000$ zeros of the zeta function, and then employed the LLL lattice reduction algorithm to solve a simultaneous approximation problem, in order to find one value of $u$ where those $70$ terms of the function $B^*_T(u)$ all lined up in a direction very near that of the positive real axis, and another value of $u$ where they all lined up very nearly in the direction of the negative real axis.
Hurst employed a similar method in one component of \cite{Hurst}, with a significantly larger number of zeros (800).

Developments by Bateman et al.\ \cite{Bateman}, Grosswald \cite{Grosswald67, Grosswald72}, Diamond \cite{Diamond}, and others established a partial version of Ingham's result connecting the independence criterion to the conclusion of having unbounded oscillations in the function $A(u)$.
Roughly, their results state that if a collection of zeros is ``sufficiently'' linearly independent, then one can conclude that the function in question has ``large'' oscillations.
To make this more precise, we require a few definitions.

Following \cite{Best}, let $\Gamma$ denote a set of positive real numbers, and given a real $T>1$ let $\Gamma'=\Gamma'(T)$ denote a subset of $\Gamma \cap [0,T]$.
For each $\gamma\in\Gamma'$, let $N_\gamma$ denote a positive integer.
We say $\Gamma'$ is \textit{$\{N_{\gamma}\}$-independent} in $\Gamma \cap [0, T]$ if two conditions hold.
First, whenever
\begin{equation}\label{ind1}
\sum_{\gamma\in\Gamma'} c_{\gamma} \gamma = 0, \quad \textrm{with } |c_{\gamma}| \leq N_{\gamma}, \quad c_{\gamma} \in \mathbb{Z},
\end{equation}
then necessarily all $c_{\gamma} = 0$.
Second, for any $\gamma^{*}\in \Gamma \cap [0, T]$, if
\begin{equation}\label{ind2}
\sum_{\gamma\in\Gamma'} c_{\gamma} \gamma = \gamma^{*}, \quad \textrm{with } |c_{\gamma}| \leq N_{\gamma},
\end{equation}
then $\gamma^{*} \in \Gamma'$, that $c_{\gamma*} = 1$, and all other $c_{\gamma} = 0$.
That is, the only linear relations with coefficients bounded by the $\{N_\gamma\}$ are the trivial ones.

In addition, if $N$ is a positive integer with the property that each $N_\gamma \geq N$ with $\{N_\gamma\}$ as above, then we say that $\Gamma'$ is \textit{$N$-independent} in $\Gamma \cap [0, T]$.
For example, Bateman et al.\ \cite{Bateman} showed that the ordinates of the first $20$ zeros of $\zeta(s)$ on the critical line are $1$-independent, in connection with their study of the Mertens conjecture.
Best and the second author \cite{Best} proved, among other things, that the first $500$ zeros are $10^{5}$-independent, in connection with work on the same problem.

We now use the machinery detailed in our earlier work \cite{MosTru} and \cite{MT2017}. 
Let $T>0$, and let $\Gamma$ denote a set of positive numbers. Define $\Gamma'= \Gamma'(T)$ as a subset of $\Gamma$ such that every $\gamma\in\Gamma'$ lies in the range $0< \gamma< T$. Finally, let $\{N_{\gamma}\}$ be a set of positive integers defined for $\gamma \in \Gamma'$.
Anderson and Stark \cite{AndStark} established the following result, which allows for explicit bounds in these oscillation problems.
\begin{theorem}[Anderson and Stark \cite{AndStark}]\label{thmAndSta}
With $A(u)$, $\Gamma$, $T$, $\Gamma'$, and $k_T(u)$ as above, if the elements of $\Gamma'$ are $\{N_{\gamma}\}$-independent in $\Gamma \cap[0,T]$, then
\begin{equation*}\label{AndS1}
\liminf_{u\rightarrow\infty} A(u) \leq r_{0} - 2\sum_{\gamma \in \Gamma'} \frac{N_{\gamma}}{N_{\gamma} + 1} k_{T}(\gamma) |r_{\gamma}|
\end{equation*}
and
\begin{equation*}\label{AndS2}
\limsup_{u\rightarrow\infty} A(u) \geq r_{0} + 2\sum_{\gamma \in \Gamma'} \frac{N_{\gamma}}{N_{\gamma} + 1} k_{T}(\gamma)  |r_{\gamma}|.
\end{equation*}
\end{theorem}
Anderson and Stark proved Theorem~\ref{thmAndSta} for the Fej\'{e}r kernel \eqref{eqnFejerkernel}; their proof was generalized by the authors \cite{MT2017} to take advantage of some latitude in the choice of this function.
In particular, the kernel of Jurkat and Peyerimhoff \eqref{eqnJPkernel} qualifies.

We next describe a method for determining values for the $N_\gamma$, given a function $A(u)$ which produces a transform $F(s)$ having simple poles on the imaginary axis as above.
We aim to determine a single value $N$ that suffices for each $\gamma$ in our set, so that we may conclude $N$-independence.
This follows the method developed in \cite{Best} and \cite{MT2017}.

Select positive integers $m$ and $n$ with $m\geq n$.
Let $\Gamma$ denote the set of ordinates of the zeros of the zeta function in the upper half-plane, and choose $T=\gamma_{m+1} -\epsilon$, with $\epsilon$ sufficiently small so that $\gamma_m<T$.
Using the kernel function \eqref{eqnJPkernel}, let $\Gamma'$ denote the subset of $\Gamma\cap[0,T]$ of size $n$ consisting of those elements $\gamma$ for which the value of $k_T(\gamma)\abs{r_\gamma}$ is largest.
Consider $m-n+1$ lattices $\Lambda_0$, \ldots, $\Lambda_{m-n}$ created by using the elements of $\Gamma\cap[0,T]$, and a parameter $b$ that governs the number of bits of precision required in the calculations.
The lattice $\Lambda_0$ has dimension $n$ and lies in $\mathbb{R}^{n+1}$: its basis is the identity matrix $I_n$ augmented by adding the row $(\rd{2^b\gamma_{a_1}}, \ldots, \rd{2^b\gamma_{a_n}})$, where $\Gamma'=\{\gamma_{a_1},\ldots,\gamma_{a_n}\}$ and $\rd{x}$ denotes the integer nearest $x$.
The subsequent lattices $\Lambda_i$ have dimension $n+1$ and lie in $\mathbb{R}^{n+2}$, and are formed in the same way, by appending the row $(\rd{2^b\gamma_{a_1}}, \ldots, \rd{2^b\gamma_{a_n}}, \rd{2^b\gamma^*})$ to $I_{n+1}$, for some $\gamma^* \in (\Gamma\cap[0,T])\backslash\Gamma'$.

If $\Gamma'$ were $N$-dependent in $\Gamma\cap[0,T]$, for some positive integer $N$, then some nontrivial relation of the form \eqref{ind1} or \eqref{ind2} would exist, with the coefficients $c_\gamma$ satisfying $\abs{c_\gamma}\leq N$.
If a nontrivial relation of the form \eqref{ind1} exists, then
\[
v=\left(c_{\gamma_{a_1}},\ldots,c_{\gamma_{a_n}},\sum_{i=1}^n c_{\gamma_{a_i}}\rd{2^b\gamma_{a_i}}\right)
\]
is a nonzero vector in $\Lambda_0$.
A short calculation determines a bound on its length:
\[
\abs{v}^2 \leq \left(\frac{n^2}{4}+n\right)N^2.
\]
Similarly, a violation of the second condition of $N$-independence \eqref{ind2} implies the existence of a nonzero vector $v$ in a corresponding lattice $\Lambda_i$ with $i>0$; here the bound is
\[
\abs{v}^2 \leq \left(\frac{n^2}{4}+n\right)N^2 + \left(\frac{n}{2}+2\right)N + \frac{5}{4}.
\]
Thus, if we can show that none of the lattices $\Lambda_0$, \ldots, $\Lambda_{m-n}$ contains a nonzero vector with length smaller than the bounds appearing here, for a particular value of $N$, then it would follow that $\Gamma'$ is $N$-independent in $\Gamma\cap[0,T]$.
To test this, we first run each basis through the LLL lattice reduction algorithm to obtain an equivalent basis that is better suited to our application---all the vectors in the original basis of $\Lambda_i$ point in nearly the same direction; in a reduced basis, the comprising vectors are in a sense more orthogonal.
Then we employ the Gram--Schmidt orthogonalization procedure on these new bases.
Since the length of a nonzero vector in a lattice is bounded below by the minimal length of a vector in the Gram--Schmidt orthogonalization of a basis of the lattice (see \cite[Lemma~1]{MT2017} for a proof), for each lattice $\Lambda_i$ we can select a value for $N_i$ to be assured that lattice $\Lambda_i$ has no such short vector.
Let $N$ be the minimal value of $N_0$, \ldots, $N_{m-n}$.
We conclude our set is $N$-independent in $\Gamma\cap[0,T]$, and we may then employ Theorem~\ref{thmAndSta} with $N_\gamma = N$ for each $\gamma\in\Gamma'$ to obtain bounds on the oscillations in $A(u)$.

\section{The Dirichlet series $h(s) = \sum_{n\geq1} (-1)^{\omega(n)}n^{-s}$}\label{secDirichletomega}

We now turn to the function $H(x)$.
To apply the Ingham--Grosswald machinery, we first need to find the Dirichlet series $h(s)$ whose coefficients are $(-1)^{\omega(n)}$.
In \eqref{banquo}, then, $e^{-u/2}H(e^u)$ then relates to $A(u)$, and $h(s)$, after suitable rescaling and shifting, relates to the  expressions in \eqref{fleance} and \eqref{escape}: $F_H(s) = h(s+1/2)/(s+1/2)$.
To that end, for $\sigma>1$ let 
\begin{equation}\label{eqnhDS}
h(s) = \sum_{n=1}^{\infty} \frac{(-1)^{\omega(n)}}{n^{s}} = \prod_{p} \left( 1 - \frac{1}{p^{s}} - \frac{1}{p^{2s}} - \cdots\right)= \zeta(s) \prod_{p} \left( 1 - \frac{2}{p^{s}}\right),
\end{equation}
where the latter equalities follow from considering the Euler product for $\zeta(s)$ and noting that $\omega(p^{e}) = 1$ for all $e\geq 1$.

We want $h(s)$ in \eqref{eqnhDS}, when continued analytically, to have only simple poles on the line $\sigma=\frac{1}{2}$, where as usual we write $s=\sigma+it$.
Note that $(1-2/p^{s}) \approx ( 1+ 2/p^{s})^{-1}$, and that 
\begin{equation}\label{run}
\zeta^{2}(s) = \prod_{p} \left( 1 + \frac{1}{p^{s}} + \frac{1}{p^{2s}} +\cdots\right)^{2} = \prod_{p} \left(1 + \frac{2}{p^{s}} + \cdots\right).
\end{equation}
Therefore the Euler product $\prod_{p} (1-2/p^{s})$ resembles the reciprocal of the left-hand side in \eqref{run}.
We write
\begin{equation}\label{tank}
h(s) = \zeta(s) \cdot \frac{1}{\zeta^{2}(s)} F_{1}(s) = \frac{F_{1}(s)}{\zeta(s)}, 
\end{equation}
and solve for $F_{1}(s)$ to see that
\begin{equation*}\label{jump}
F_{1}(s) = h(s)\zeta(s) = \prod_{p} \left( 1 -\frac{1}{p^{2s}} - \cdots \right).
\end{equation*}
We therefore have an expression in \eqref{tank} giving $h(s)$ in terms of $F_{1}(s)$, the latter being analytic for $\sigma > 1/2$.
We now continue, noting that $F_{1}(s)$ resembles the Euler product of $\zeta(2s)^{-1}$.
Writing $F_{1}(s) = F_{2}(s)/\zeta(2s)$, and solving for $F_{2}(s)$ we have
\begin{equation*}\label{eqnF2}
h(s) = \frac{F_{2}(s)}{\zeta(s)\zeta(2s)}, \quad F_{2}(s) = \prod_{p} \left( 1 - \frac{2}{p^{3s}} - \cdots \right),
\end{equation*}
whence we see that $F_{2}(s)$ is analytic for $\sigma>1/3$.
We continue in this way, each time using a suitable power of $\zeta(ks)$ to mimic the preceding $F_{k}(s)$.
We find that
\begin{equation*}\label{baton}
h(s) = \frac{F_{k}(s)}{\zeta(s) \zeta(2s) \zeta^{2}(3s) \cdots \zeta^{a_{k}}(ks)},
\end{equation*}
with $F_{k}(s)$ analytic for $\sigma > \frac{1}{k+1}$ and $k\geq 1$, where the sequence\footnote{This sequence appears in the OEIS as A059966.} $\{a_{k}\}_{k\geq 1}$ begins $1$, $1$, $2$, $3$, $6$, \ldots, and is defined by the formal product
\[
\prod_{k=1}^{\infty} \left(1-q^{k}\right)^{a_{k}} = 1 - q - q^{2} - q^{3}  - \cdots.
\]

Choosing $k=6$ we have
\begin{equation}\label{turkey}
h(s) = \frac{F_{6}(s)}{\zeta(s)\zeta(2s)\zeta^{2}(3s)\zeta^{3}(4s)\zeta^{6}(5s)\zeta^{9}(6s)},
\end{equation}
where $F_{6}(s) = \prod_{p} (1 - 18p^{-7s} - 30 p^{-8s} - 56 p^{-9s} - \cdots)$.
Since $F_{6}(s)$ converges for $\sigma > 1/7$ we can see that $h(s)$ has zeros at $s=1$, $s=1/2$, $1/3$, \ldots, $1/6$ and poles at $s = \rho_n/k$ for $k=1$, $2$, $3$. Indeed, for $k=3$ there are double poles of $h(s)$.

This gives rise to some lower-order terms when we apply Ingham's method.
Although the poles on $\sigma = 1/2$ drive the oscillations, we have some minor contributions from poles at $\sigma = 1/4$, $1/6$, \ldots\,.
This contrasts markedly with the Dirichlet series $ \sum_{n\geq 1} (-1)^{\Omega(n)}n^{-s}=\zeta(2s)/\zeta(s)$.
Under our working assumption of the Riemann hypothesis, the function $\zeta(2s)/\zeta(s)$ has no poles to the left of the `driving terms' $s = \frac{1}{2} + i\gamma$.
Moreover, $h(1) = 0$, which is what van de Lune and Dressler \cite{Dressler} showed.

Let us examine the residues on $\sigma = 1/2$.
There, we have only simple poles, and the residue of $1/\zeta(s)$ at $s=\rho_n$ is just $\zeta'(\rho_n)^{-1}$.
Therefore, since all other factors in $h(s)$ in \eqref{turkey} are analytic at $s=\rho_n$, we find that the residue of $h(s)/s$ at $s=\frac{1}{2}+i\gamma_n$ is
\begin{equation}\label{carson}
r_{\gamma_n} = \frac{F_{6}(\rho_n)}{\zeta(2\rho_n)\zeta^{2}(3\rho_n)\zeta^{3}(4\rho_n)\zeta^{6}(5\rho_n)\zeta^{9}(6\rho_n)} \cdot \frac{1}{\rho_n\zeta'(\rho_n)}.
\end{equation}

We note that to apply Ingham's method we need to verify that $\sum_{n\geq 1} |r_{\gamma_{n}}|$ diverges.
In the case of $M(x)$, this follows from Titchmarsh \cite[Section 14.27]{Touchy}; Ingham \cite[Theorem 2]{Ingham1942} reproved this, and extended the result to cover $L(x)$.
In our case we follow Ingham's arguments on \cite[p.\ 317]{Ingham1942}, noting that we merely need a decent lower bound on $\zeta(s)$ when $\sigma\geq 1$, where $s = \sigma + it$.
Although stronger bounds are known, even the classical result that $|\zeta(s)|^{-1} = O(\log t)$ for $\sigma \geq 1$ --- see, e.g., \cite[Section 3.11]{Touchy} --- suffices.  Ingham's proof therefore applies to the divergence of $\sum_{n\geq 1} |r_{\gamma_{n}}|$ with $r_{\gamma_{n}}$ in \eqref{carson}.

We can also add in the contributions from the residues on $\sigma = 1/4$.
There, we have only simple poles, and we have that the residue of $1/\zeta(2s)$ at $s=\rho_n/2$ is $1/2\zeta'(\rho_n)$.
Therefore, since all other factors in $h(s)$ in \eqref{turkey} are analytic at $s=\rho_n/2$ we have that the residue here is
\begin{equation}\label{boise}
\frac{F_{6}(\rho_n/2)}{\zeta(\rho_n/2)\zeta^{2}(3\rho_n/2)\zeta^{3}(2\rho_n)\zeta^{6}(5\rho_n/2)\zeta^{9}(3\rho_n)} \cdot \frac{1}{\rho_n\zeta'(\rho_n)}.
\end{equation}
We could therefore add in \eqref{boise} to try to correct for the poles due to $\zeta(2s)$ in $h(s)$.
Terms from \eqref{boise} will be weighted by $e^{u/4}$, which is smaller than the weighting $e^{u/2}$ applied to the terms in \eqref{carson}.
We could continue in this way, adding in the contribution from the multiple poles at $s=\rho_n/3$, $s=\rho_n/4$, etc. It is the minor contribution from these poles, and those at $s=\rho_{n}/2$, that is the source of the initial bias, which eventually dissipates --- cf.\ Figures \ref{figMLH} and \ref{figBigH}.

\section{Qualitative analysis}\label{secQualAn}

Using Perron's formula, and following the procedure in \cite{MosTru}, we have
\begin{equation}\label{lure}
\frac{L(e^u)}{e^{u/2}} = \frac{1}{\zeta(1/2)} + \sum_{|\gamma_{n}|\leq T} \frac{\zeta(2\rho_{n}) e^{i\gamma_{n}u}}{\rho_{n}\zeta'(\rho_{n})} + E(u, T),
\end{equation}
where $E(u, T)$ is an error term for which $\lim_{u\rightarrow\infty}E(u, T) = 0$.
In a loose sense, the fact that $\zeta(1/2)<0$ accounts for the negative bias of $L(x)$.
This value is evident in the plot of $e^{-u/2}L(e^u)$ in Figure~\ref{figMLH}, as this curve oscillates about the displayed horizontal line $y=1/\zeta(1/2)=-0.6847\ldots$\,.
Similarly, the fact that infinitely often the sum over the zeros exceeds this constant term accounts for the infinitude of integers that provide a negative answer to P\'{o}lya's question.

More formally, Humphries \cite{HumphriesJNT}, following the work of Rubinstein and Sarnak \cite{RS}, examined the logarithmic density of those $x$ for which $L(x)\leq 0$. Under the Riemann hypothesis, the linear independence hypothesis, and the discrete moment conjecture\footnote{This is the statement that $\sum_{0<\gamma_{n}<T} |\zeta'(\rho_{n})|^{-2}\ll T.$}, Humphries was able to prove that $L(x)\leq 0$ at least half of the time.
See \cite[Thm.\ 1.5]{HumphriesJNT} for more details. We also mention here computations performed by Brent, and cited in \cite[p.\ 577]{HumphriesJNT}, which suggest that the limiting logarithmic density of the set of $y$ for which $L(y)\leq 0$ is approximately $0.99988$.
 
Applying the Dirichlet series for $H(x)$ from the previous section, and working through the same procedure, we have
\begin{equation}\label{trout}
\frac{H(u)}{e^{u/2}} = \sum_{|\gamma_{n}|\leq T} \frac{F_6(\rho_n) e^{i\gamma_{n}u}}{\rho_{n}\zeta'(\rho_n)\zeta(2\rho_n)\zeta^{2}(3\rho_n)\zeta^{3}(4\rho_n)\zeta^{6}(5\rho_n)\zeta^{9}(6\rho_n)} + E(u, T).
\end{equation}
Note that, unlike \eqref{lure}, there is no leading term in \eqref{trout}, since $r_0=0$ in this case, and hence there is no bias.

We now compare the estimate of \eqref{trout} with the empirical data.
Figure~\ref{figHsamp} exhibits the actual values of $e^{-u/2}H(e^u)$ over $24\leq u\leq 27$, so approximately $2.5\cdot10^{10}\leq x\leq 5\cdot10^{11}$, and Figure~\ref{figHest} displays the main term of \eqref{trout} with $T=5000$ over this same range.
These plots show a very good match for our estimate in \eqref{trout} in the \textit{local} behavior of $H(x)$, in its match of the oscillations, but our estimate does not capture some \textit{global} structure, in a carrier wave that these oscillations are centered on.
We found a similar phenomenon in \cite{MT2017}, in the investigation of $L_{1/2}(x) = \sum_{n\leq x} \lambda(n)/\sqrt{n}$.
Here, a pole of order 2 occurred at the origin in the Laplace transform of $L_{1/2}(e^u)$, and its presence caused the oscillations to center on a line with slope determined by the coefficient of $1/s^2$ in the Laurent series for the transform, expanded about the origin.
No poles of order larger than $1$ arise in the analysis for $M(x)$ and $L(x)$, and the empirical plots for these functions (suitably normalized as in Figure~\ref{figMLH}) match our corresponding analytic estimates involving the zeros of the zeta function, like \eqref{lure} for P\'olya's function, very well.
(See for example \cite[Fig.\ 1]{Borwein}.)

\begin{figure}[tb]
\centering
\includegraphics[width=4in]{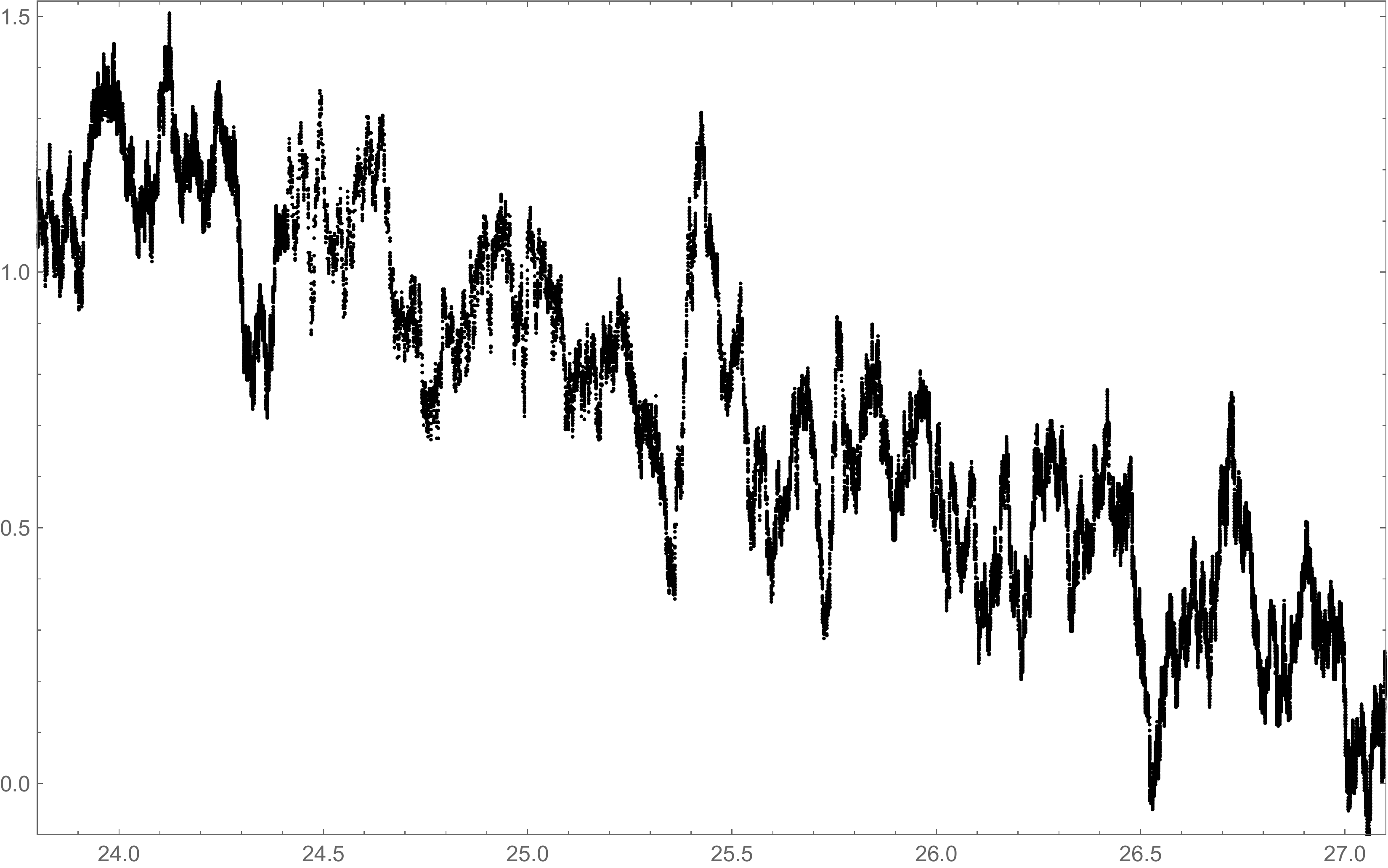}
\caption{Actual sampled values of $e^{-u/2}H(e^u)$.}\label{figHsamp}
\end{figure}

\begin{figure}[tb]
\centering
\includegraphics[width=4in]{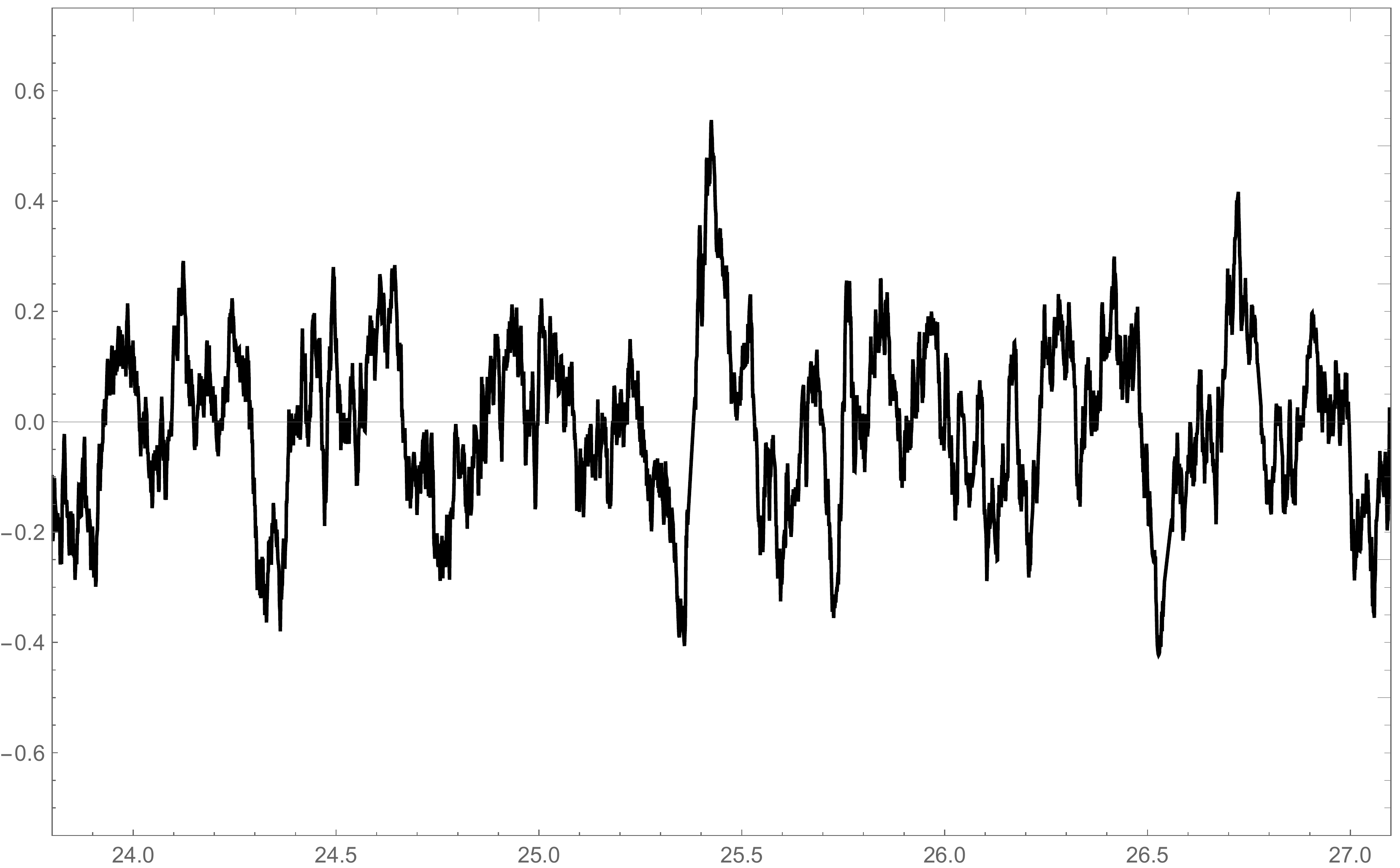}
\caption{Estimate of $e^{-u/2}H(e^u)$ using the sum in \eqref{trout} with $T=5000$.}\label{figHest}
\end{figure}

However, our Dirichlet series $h(s)$ has many poles of higher order.
Unlike the analysis for $L_{1/2}(x)$, these poles occur \textit{behind} the critical line.
From \eqref{turkey}, it is evident that each nontrivial zero of the zeta function $\rho$ gives rise to a pole of order $2$ at $\rho/3$, one of order $3$ at $\rho/4$, etc., and each of these zeros produces a contribution that perturbs the local oscillations.
These contributions will attenuate as $u$ grows: the terms from the poles of order $2$ on the $\sigma=1/4$ line inherit a  damping factor of $e^{-u/4}$, and those from the poles of order $3$ on the $\sigma=1/6$ line would be assessed with a factor of $e^{-u/3}$.
We surmise that the net effect of these contributions produces the global structure evident in the plots of $e^{-u/2}H(e^u)$ in Figures~\ref{figMLH} and~\ref{figBigH} for small values of $u$, but we leave a more precise modeling here to future research.

We expect then that the additional global structure for $e^{-u/2}H(e^u)$ that is evident in Figures~\ref{figMLH} and~\ref{figBigH} will dissipate as $u$ grows large, and eventually the oscillations will center on the $0$ line.
We remark also that residues from these higher-order poles, like \eqref{boise} for a pole of order $2$ on the $\sigma=1/4$ line, will also produce some contribution, but these are insignificant even in the range plotted in Figure~\ref{figHsamp}.

\section{Calculations}\label{secCalculations}

We describe two large calculations required for this work.
First, we provide the details of our establishment of weak independence of certain zeros of the Riemann zeta function, to furnish a proof of Theorem~\ref{thmOscillate}.
Second, we describe a sieving procedure employed to generate the values of $\xi(n)$ that were used to produce the plot of $e^{-u/2}H(e^u)$ in Figure~\ref{figBigH}.

\subsection{Establishing weak independence}\label{subsecWeakInd}
The weak independence method has been applied to both the Mertens function $M(x)$ and P\'olya's function $L(x)$.
For the former, Best and the second author \cite{Best} selected $n=500$, $m=2000$, and precision parameter $b=30000$, and so reduced $1501$ lattices of dimension $500$ or $501$.
They chose a modest value for the parameter $\delta$ in the LLL algorithm \cite{LLL}: this parameter governs the amount of work performed by the algorithm and affects the quality of the obtained basis.
This parameter needs to lie in the interval $(1/4,1)$, with smaller choices producing faster run times but worse guarantees on the quality of the basis.
With the parameters above and using $\delta=0.3$, the authors obtained that $N=4976$ suffices, and using their 500 best residues they proved \eqref{eqnOscMerBT}.

For P\'olya's function $L(x)$, the authors employed a different tack in \cite{MT2017} to obtain the bounds \eqref{eqnOscPolMT}, using a larger number of smaller lattices ($n=250$, $m=3700$), at a lower precision ($b=6600$ for most lattices, $b=6950$ for some) but a higher value of $\delta= 0.99$, in order to minimize computation time while obtaining the desired result that $L(x)>\sqrt{x}$ infinitely often.
The $N$ value obtained here was $3100$.
We mention briefly that these bounds improved some prior results: Anderson and Stark \cite{AndStark} had $0.02348$ as a lower bound on the upper oscillation by using data from \cite{Lehman}.
This was improved to $0.06186$ in \cite{Borwein}.
For the lower oscillation, Humphries \cite{HumphriesJNT} had noted that the method of \cite{AndStark} and the data from \cite{Borwein} produces $-1.3829$.

The weak independence result for $L(x)$ produces an oscillation bound in the $H(x)$ problem immediately.
By simply using the residues for $h(s)$ arising from the zeros of the zeta function which were optimal for the $L(x)$ problem, one obtains $H(x)>1.613\sqrt{x}$ and $H(x)<-1.613\sqrt{x}$ by using this existing weak independence result.
However, we can do better with a new calculation, since the zeros that were optimal for $L(x)$  are unlikely to be optimal also for $H(x)$.
For this problem, we begin by estimating the time required to obtain a particular oscillation bound $B$ in terms of our parameters.
Setting a goal of $N=3000$ or better, for each potential value of $n$ we can determine the smallest $m$ needed to obtain a value for $\sum_{\gamma\in\Gamma'}k_T(\gamma)\abs{r_\gamma}$ that is at least as large as $(N+1)B/N$, and we find empirically that the precision $b$ required grows linearly in $n$.
This allows us to obtain estimates for run times for each configuration, and we computed estimates for $B=1.7$, $B=1.75$, and $B=1.8$.
For $B=1.7$, we found that choosing $n=239$, $m=2365$, $b=6400$ is optimal, and using $\epsilon=10^{-10}$ and $\delta=0.99$ we performed these computations on the $m-n+1=2127$ generated matrices in a large, distributed computation.
We conclude that $N=3950$ suffices, and this allows us to confirm Theorem~\ref{thmOscillate}.
The indices of the $n=239$ best zeros of $\zeta(s)$ for this calculation (relative to $m=2365$) appear in Figure~\ref{figBestZerosH}.

\begin{figure}[tb]
\small
\begin{quote}
1--72, 74--76, 78--116, 118--145, 147--170, 172--178, 180, 182--188, 190--194,
196--198, 200--202, 204, 206, 207, 209, 212--214, 216, 217, 219, 222, 224, 225,
230, 232--234, 237, 238, 240, 242--245, 248--250, 253, 257, 263, 265, 268, 269,
275, 282, 283, 290, 298, 299, 301, 311, 314--316, 327, 340, 364.
\end{quote}
\caption{Indices of the $239$ zeros of $\zeta(s)$ employed in the calculation for $H(x)$.}\label{figBestZerosH}
\end{figure}

The lattice reduction and Gram--Schmidt calculations were performed on Raijin, a high-performance computer managed by NCI Australia.
Our $m-n+1=2127$ jobs ran on Intel Xeon Sandy Bridge and Broadwell $2.6$ GHz processors, and most required about 110 minutes to complete.
We estimate that establishing an oscillation bound of $B=1.75$ for $H(x)$ would require $2.4$ times the work that was needed for $B=1.7$, using $n=263$ and $m=2647$, and $B=1.8$ would require $5.9$ times the effort, with $n=291$ and $m=2872$.
The precision parameter $b$ would need to increase as well here, by approximately $50$ bits per added dimension over $n=239$.

It is entirely possible to use a kernel different from \eqref{eqnJPkernel}.
While we have not pursued this, we note that Odlyzko and te Riele explored this idea in \cite[pp.\ 150--151]{OTR}.
Best and the second author presented some computational data in \cite[Fig.~2]{Best} showing that even the invalid choice of $k_{T}(\gamma) =1$ produces only minimal improvements.

\subsection{Computing $H(x)$ over a large interval}\label{subsecSieve}

A sieving strategy was employed to compute $\xi(n)$ for $1\leq n\leq 1.5\cdot10^{15}$. These values were  used to create the plot of $H(x)/\sqrt{x}$ in Figure~\ref{figBigH}.
This method is similar to the one employed in \cite[Algs.\ 1 and 2]{Borwein} for computing the values of $\lambda(n)$ over a large interval.
It employs a large table of values of $\xi(n)$ in order to cut the amount of factoring that is required: if $p\mid n$ and $\xi(n/p)$ is stored in our table, then we can set $\xi(n)$ to either $\xi(n/p)$ or $-\xi(n/p)$ depending on whether or not $p \mid n/p$.
The first step then is constructing this table.
In order to optimize the utility of the table, we store $\xi(n)$ only for integers $n$ that are relatively prime to $30$, up to a particular bound $N-1$.
It is convenient to select $N$ to be a multiple of $30$.
Note just one bit is needed to store each $\xi(n)$ value, so a table of size $N$ requires $N/30$ bytes of storage.
The table is constructed by using a boot-strapping strategy: one begins with a table size of one byte representing $\xi(n)$ for $n\in\{1,7,11,13,17,19,23,29\}$, then employs the sieving procedure described below to double the table size repeatedly until reaching the desired final size $N$.

With this table in hand, we use the following procedure to compute $\xi(n)$ over all integers in a given interval $[a,b)$ with positive integers $a<b$.
First, we allocate a block of size $2(b-a)$ bits and fill it with zeros to indicate that $\xi(n)$ is presently unknown for each integer in this interval.
(Two bits are needed for each integer, to indicate if $\xi(n)=\pm1$, or if $\xi(n)$ is presently unknown.)
Next, for each prime $p$ with $b/(N-1)\leq p\leq\sqrt{b}$, we use a sieving strategy to visit each multiple of $p$ in $[a,b]$.
Each such integer $n$ has $n/p\leq N-1$, so $\xi(n/p)$ can be determined by using the table, after removing all factors of $2$, $3$, and $5$, and keeping track of how many of these small primes divide $n$.
After this, for each prime $p$ from $\ceiling{b/(N-1)}-1$ down to $7$, we perform the same sieving strategy, and at each integer $n$ divisible by $p$, we remove any factors of $2$, $3$, and $5$, and check if the resulting cofactor $n'$ is less than $N$.
If it is, we can determine $\xi(n)$ with the information at hand.
If it is not, then we must use trial division with primes starting at $q=7$ to search for any additional prime factors of $n'$, to reduce the cofactor still further.
We halt either when the remaining cofactor is less than $N$, or when our trial division procedure reaches the prime $p$, in which case this remaining cofactor is prime.
In all cases we may deduce the value of $\xi(n)$.
Next, we scan our list for integers in $[a,b]$ where $\xi(n)$ remains unknown---these all have the form $2^r3^s5^tp^e$ with $r$, $s$, $t$ all nonnegative integers, $p>5$ prime, and $e=0$ or $1$, and set $\xi(n)$ appropriately.
A final pass computes $\sum_{n=a}^{b-1} \xi(n)$, occasionally printing the value of the partial sum, to facilitate the construction of the plot in Figure~\ref{figBigH}.

This sieving procedure is distributed across many nodes in a cluster, each with its own copy of the table of values of $\xi(n)$ for small $n$.
This is organized so that each node handles a number of blocks in succession, on the order of $10^4$ blocks of size $10^7$ or $10^8$.
We employed values of $N$ ranging from approximately $N=4\cdot10^{10}$ (requiring about $1.24$ GB) to $N=1.5\cdot10^{12}$ ($46.6$ GB), as larger tables facilitated sieving intervals of larger integers.
Computations were performed on Raijin at NCI Australia, and on a cluster at Davidson College with Intel Xeon Sandy Bridge $3.0$ GHz processors.

\section{Open problems}\label{secOpen}

The true orders of growth of $M(x)$ and $L(x)$ have attracted substantial attention, but these problems remain open.
An unpublished conjecture of Gonek (see Ng \cite{Ng}) for $M(x)$ can be extended to cover $L(x)$ as well.
This conjecture states that there exists a positive number $A$ for which
\begin{equation*}\label{soup}
\liminf_{x\rightarrow\infty} \frac{L(x)}{x^{1/2} (\log\log\log x)^{5/4}} = -A, \quad \limsup_{x\rightarrow\infty} \frac{L(x)}{x^{1/2} (\log\log\log x)^{5/4}} = A,
\end{equation*}
and similarly for $M(x)$. 
Almost certainly the same result should be expected for $H(x)$.
These conjectures seem out of reach at present.

The explicit formula in (\ref{lure}) with its negative leading term, indicates a bias for sums of $\lambda(n)$, whereas (\ref{trout}) indicates no bias for sums of $(-1)^{\omega(n)}$. Humphries' work, mentioned at the start of Section~\ref{secQualAn}, may be complemented by work of Brent and van de Lune \cite{BrentL} who showed, that, in some sense, $\lambda(n)$ is negative on the average. More specifically, they showed that there is a positive constant $c$ for which, for every fixed $N\geq 1$ we have
\begin{equation}\label{rush}
\sum_{n=1}^{\infty} \frac{\lambda(n)}{e^{n\pi x} +1} = -\frac{c}{\sqrt{x}} + \frac{1}{2} + O(x^{N}),
\end{equation}
as $x\rightarrow 0^{+}$. It would be interesting to see whether a result similar to (\ref{rush}) holds with sums of $(-1)^{\omega(n)}$, or whether the negative leading term disappears.

Another question of wide interest originated with Chowla \cite{Chowla}, who conjectured that for each $k\geq1$, each of the $2^k$ sequences in $\{-1,1\}^k$ should occur with density $2^{-k}$ among all sequences of the form $\lambda(n)$, $\lambda(n+1)$, \ldots, $\lambda(n+k-1)$.
This is known to be true for $k=1$, as this is equivalent to the Prime Number Theorem, but it is not known for any $k>1$. 

For $k=2$, a problem posed by Graham and Hensley \cite{GH} and solved by them and by Propp in \cite{GHsolve} shows that the pattern $\lambda(n) = \lambda(n+1)$ occurs with density at least $1/4$. This was improved by Tao \cite{TaoSolo1} who showed that the patterns $\lambda(n) = \lambda(n+1)=1$ and $\lambda(n) = \lambda(n+1)=-1$ each occur with logarithmic density $1/4$.

A partial result on Chowla's conjecture for $k\leq 3$ was obtained by Hildebrand \cite{Hildebrand} by elementary means.
Hildebrand proved that each of the possible sign patterns of size $k=3$ occur infinitely often.
Trying to show that each possible triple of sign patterns occurs with density precisely $1/8$ seems very difficult.
However, significant progress has recently been made by Matom\"{a}ki, Radziwi\l\l, and Tao \cite{MRT}, who proved that each sign pattern occurs with positive density. 
They note that their method may be able to give an effective estimate for this density.
Can this be extended to the other sign patterns?
Can this technique, or Hildebrand's result, be applied to the sequence $\xi(n) = (-1)^{\omega(n)}$?

Finally, we note a recent paper by Tao and Ter\"av\"ainen \cite{Tao2}, who prove two results pertinent to this discussion.
First, they establish that the triple 
\[
\omega(n), \; \omega(n+1), \; \omega(n+2) \pmod 3
\]
assumes all $27$ possible combinations with positive density.
They also prove that at least $24$ of the $32$ possible sign patterns of quintuples of $\lambda(n) = (-1)^{\Omega(n)}$ occur with positive density.
This improves on their results in \cite{TT2}, where they also showed that each sign pattern of length 3 occurs with logarithmic density $1/8$.

\section*{Acknowledgments}

We thank NCI Australia, UNSW Canberra, and Davidson College for computational resources.
This research was undertaken with the assistance of resources and services from the National Computational Infrastructure (NCI), which is supported by the Australian Government.

We are grateful to the anonymous referees for very useful feedback and wish to thank the following colleagues for fruitful discussions: Peter Humphries, Kaisa Matom\"{a}ki, Hugh Montgomery, Nathan Ng, Maksym Radziwi\l\l, Terry Tao, and Joni Ter\"{a}v\"{a}inen.


\bibliographystyle{amsplain}

\begin{thebibliography}{10}

\bibitem{Alladi}
K. Alladi. 
\newblock On the probability that $n$ and $\Omega(n)$ are relatively prime. \newblock {\em Fib. Quart.}, 19:228--233, 1981.

\bibitem{AndStark}
R.~J. Anderson and H.~M. Stark.
\newblock Oscillation theorems.
\newblock {\em Lecture Notes in Math.}, 899:79--106, 1981.

\bibitem{Bateman}
P.~T. Bateman, J.~W. Brown, R.~S. Hall, K.~E. Kloss, and R.~M. Stemmler.
\newblock Linear relations connecting the imaginary parts of the zeros of the zeta function.
\newblock{\em Computers in Number Theory} (Proc. Sci. Res. Council Atlas Sympos. No. 2, Oxford, 1969), 1971, pp.\ 11--19.

\bibitem{Best}
D.~G. Best and T.~S. Trudgian.
\newblock Linear relations of zeroes of the zeta-function.
\newblock {\em Math. Comp.}, 84(294): 2047--2058, 2015.

\bibitem{Borwein}
P.~Borwein, R.~Ferguson, and M.~J. Mossinghoff.
\newblock Sign changes in sums of the {L}iouville function.
\newblock {\em Math. Comp.}, 77(263):1681--1694, 2008.

\bibitem{BrentL}
R.~P. Brent and J. van de Lune.
\newblock A note on P\'{o}lya's observation concerning Liouville's function.
\newblock {\em Herman J.~J. te Riele Liber Amicorum}, CWI, Amsterdam, Dec. 2011, 92--97.

\bibitem{Chowla}
S. Chowla.
\newblock {\em The Riemann Hypothesis and Hilbert's Tenth Problem}, Gordon and Breach, New York, 1965.

\bibitem{Cooper}
C.~N. Cooper and R.~E. Kennedy. 
\newblock Chebyshev's inequality and natural density. \newblock {\em Amer. Math. Monthly}, 96:118--124, 1989.

\bibitem{Diamond}
H.~G. Diamond.
\newblock Two oscillation theorems.
\newblock In {\em The theory of arithmetic functions} (Proc. Conf., Western Michigan Univ., Kalamazoo, Mich., 1971), 1972, Springer, Berlin, pp.\ 113--118.

\bibitem{Zim}
J. Detrey, P.-J. Spaenlehauer, and P. Zimmermann.
\newblock Computing the $\rho$ constant.
\newblock Available at \url{https://members.loria.fr/PJSpaenlehauer/data/notes/rho.pdf}, 2016.

\bibitem{MacBeth}
R.~L. Duncan.
\newblock On the factorization of integers. 
\newblock {\em Proc. Amer. Math. Soc.}, 25:191--192, 1970.

\bibitem{Ivic82}
P. Erd\H{o}s and A. Ivi\'{c}.
\newblock On sums involving reciprocals of certain arithmetical functions.
\newblock {\em Publ. Inst. Math. (Beograd) (N.S.)}, 32(46):49--56, 1982.

\bibitem{EPom}
P. Erd\H{o}s and C. Pomerance. 
\newblock On a theorem of Besicovitch: values of arithmetic functions that divide their arguments. \newblock {\em Indian J. Math.}, 32(3):279--287, 1990.

\bibitem{EPS2}
P. Erd\H{o}s, C. Pomerance, and A. S\'{a}rk\"{o}zy.
\newblock On locally repeated values of certain arithmetic functions. II.
\newblock {\em Acta Math. Hungar.}, 49:215--259, 1987.

\bibitem{EPS3}
P. Erd\H{o}s, C. Pomerance, and A. S\'{a}rk\"{o}zy.
\newblock On locally repeated values of certain arithmetic functions. III.
\newblock {\em Proc. Amer. Math. Soc.}, 101:1--7, 1987.

\bibitem{card}
D.~A. Goldston, S.~W. Graham, J. Pintz, and C.~Y. Y{\i}ld{\i}r{\i}m.
\newblock Small gaps between almost primes, the parity problem, and some conjectures of Erd\H{o}s on consecutive integers.
\newblock{\em Int. Math. Res. Not. IMRN}, 7:1439--1450, 2011.

\bibitem{GH}
S.~W. Graham and D. Hensley.
\newblock Problem E3025.
\newblock{\em Amer. Math. Monthly}, 90(10): 706, 1983.

\bibitem{Grosswald67}
E.~Grosswald.
\newblock Oscillation theorems of arithmetical functions.
\newblock {\em Trans. Amer. Math. Soc.}, 126:1--28, 1967.

\bibitem{Grosswald72}
E.~Grosswald.
\newblock Oscillation theorems.
\newblock {\em Lecture Notes in Math.}, 251:141--168, 1972.


\bibitem{HRam}
G.~H. Hardy and S. Ramanujan.
\newblock The normal number of prime factors of a number $n$.
\newblock{\em Quart. J. Math.}, 48:76--92, 1917.

\bibitem{Haselgrove}
C.~B. Haselgrove.
\newblock A disproof of a conjecture of P\'{o}lya.
\newblock{\em Mathematika}, 5:141--145, 1958.

\bibitem{RHB}
D.~R. Heath-Brown.
\newblock The divisor function at consecutive integers.
\newblock {\em Mathematika}, 31:141--149, 1984.

\bibitem{Hildebrand}
A.~J. Hildebrand.
\newblock On consecutive values of the Liovuille function.
\newblock {\em Enseign. Math. (2)}, 32(3-4):219--226, 1986.

\bibitem{HumphriesJNT}
P.~Humphries.
\newblock The distribution of weighted sums of the Liouville function and P\'{o}lya's conjecture.
\newblock {\em J. Number Theory}, 133(2):545--582, 2013.

\bibitem{Hurst}
G. Hurst.
\newblock Computations of the Mertens function and improved bounds on the Mertens conjecture.
\newblock {\em Math. Comp.}, 87(310):1013--1028, 2018.

\bibitem{Ingham1942}
A.~E. Ingham.
\newblock On two conjectures in the theory of numbers.
\newblock {\em Amer. J. Math.}, 64(1):313--319, 1942.


\bibitem{Ivic88}
A. Ivi\'{c}.
\newblock The distribution of quotients of small and large additive functions. 
\newblock {\em Boll. Un. Mat. Ital. B(7)}, 2:79--97, 1988.

\bibitem{Kon80}
A. Ivi\'{c} and J.-M. De Koninck.
\newblock Topics in arithmetical functions.
\newblock {\em Asymptotic Formulae for Sums of Reciprocals of Arithmetical Functions and Related Results.} Notas de Mathem\'{a}tica, (72), North-Holland Publishing Co., Amsterdam--New York, 1980.

\bibitem{JP}
W. Jurkat and A. Peyerimhoff.
\newblock A constructive approach to Kronecker approximations and its application to the Mertens conjecture.
\newblock {\em J. Reine Angew. Math.}, 286/287:322--340, 1976.

\bibitem{Kon72}
J.-M. De Koninck.
\newblock On a class of arithmetical functions.
\newblock {\em Duke Math. J.}, 39:807--818, 1972.

\bibitem{Kon74}
J.-M. De Koninck.
\newblock Sums of quotients of additive functions.
\newblock {\em Proc. Amer. Math. Soc.}, 44:35--38, 1974.

\bibitem{Dressler}
J. van de Lune and R.~E. Dressler.
\newblock Some theorems concerning the number theoretic function $\omega(n)$.
\newblock {\em J. Reine Angew. Math.}, 277:117--119, 1975.

\bibitem{Lehman}
R.~S. Lehman.
\newblock On Liouville's function.
\newblock {\em Math. Comp.}, 14:311--320, 1960.

\bibitem{LLL}
A.~K. Lenstra, H.~W. Lenstra, Jr., and L. Lov\'asz.
\newblock Factoring polynomials with rational coefficients.
\newblock {\em Math. Ann.}, 261(4):515--534, 1982.

\bibitem{MRT}
K. Matomaki, M. Radziwi\l\l, and T. Tao.
\newblock Sign patterns of the Liouville and M\"{o}bius functions.
\newblock {\em Forum Math. Sigma}, 4(e14):44pp., 2016.

\bibitem{Mertens}
F. Mertens.
\newblock \"Uber eine zahlentheoretische Funktion.
\newblock {\em Sitzungsberichte Akad. Wien}, 106(Abt. 2a):761-830, 1897.

\bibitem{MV}
H.~L. Montgomery and R.~C. Vaughan. \newblock {\em Multiplicative Number Theory. I. Classical Theory.}
\newblock Cambridge Stud. Adv. Math., 97. Cambridge Univ.\ Press, Cambridge, 2007.

\bibitem{Moree}
P. Moree.
\newblock Approximation of singular series and automata.
\newblock {\em Manuscripta Math.}, 101(3):385--399, 2000.

\bibitem{MosTru}
M.~J. Mossinghoff and T.~S. Trudgian.
\newblock Between the problems of {P}\'{o}lya and {T}ur\'{a}n.
\newblock {\em J. Aust. Math. Soc.}, 93(1-2):157--171, 2012.

\bibitem{MT2017}
M.~J. Mossinghoff and T.~S. Trudgian.
\newblock The Liouville function and the Riemann hypothesis.
\newblock In {\em Exploring the Riemann Zeta Function: 190 Years from Riemann's Birth}, H.~L. Montgomery et al. (eds.), 2017, Springer, Cham, pp.\ 201--221.

\bibitem{Ng}
N. Ng.
\newblock The distribution of the summatory function of the M\"{o}bius function.
\newblock {\em Proc. London Math. Soc. (3)}, 89:361--389, 2004.

\bibitem{OTR}
A.~M. Odlyzko and H.~J.~J. te Riele.
\newblock Disproof of the Mertens conjecture.
\newblock {\em J. Reine Angew. Math.}, 357:138--160, 1985.

\bibitem{Polya}
G. P\'{o}lya.
\newblock Verschiedene Bemerkungen zur Zahlentheorie.
\newblock{\em Jahresber. Deutsch. Math.-Verein.}, 28:31--40, 1919.

\bibitem{GHsolve}
J. Propp.
\newblock Solution to Problem E3025.
\newblock{\em Amer. Math. Monthly}, 93(8):655, 1986.

\bibitem{Ren}
A. R\'{e}nyi.
\newblock On the density of certain sequences of integers.
\newblock{\em Acad. Serbe Sci. Publ. Inst. Math.}, 8:157--162, 1955.

\bibitem{Rob}
R.~L. Robinson.
\newblock An estimate for the enumerative functions of certain sets of integers.
\newblock{\em Proc. Amer. Math. Soc.}, 16:232--237, 1966.

\bibitem{RS}
M. Rubinstein and P. Sarnak.
\newblock Chebyshev's bias.
\newblock{\em Experiment. Math.}, 3(3):173--197, 1994.

\bibitem{Sandy}
J. S\'{a}ndor, D.~S. Mitrinovi\'{c}, and B. Crstici.
\newblock {\em Handbook of Number Theory. I.}
\newblock Springer, Dordrecht, 2006.

\bibitem{Puchta}
J.-C. Schlage-Puchta.
\newblock The equation $\omega(n) = \omega(n+1)$.
\newblock{\em Mathematika}, 50(1-2):99--101, 2003.

\bibitem{Schwarz}
W. Schwarz.
\newblock A remark on multiplicative functions.
\newblock {\em Bull. London Math. Soc.}, 4:136--140, 1972.


\bibitem{TaoSolo1}
T. Tao.
\newblock The logarithmically averaged Chowla and Elliott conjectures for two-point correlations.
\newblock {\em Forum Math. Pi}, 4(e8):36pp., 2016.

\bibitem{Tao2}
T. Tao and J. Ter\"{a}v\"{a}inen.
\newblock Value patterns of multiplicative functions and related sequences.
\newblock {\em Forum Math. Sigma}, 7(e33):55pp., 2019.

\bibitem{TT2}
T. Tao and J. Ter\"{a}v\"{a}inen.
\newblock The structure of logarithmically averaged correlations of multiplicative functions, with applications to the Chowla and Elliot conjectures.
\newblock  {\em Duke Math. J.}, 168(11):1977-2027, 2019.

\bibitem{Touchy}
E.~C. Titchmarsh.
\newblock {\em The Theory of the Riemann Zeta-Function}, 2nd ed., Oxford Univ.\ Press, New York, 1986.

\bibitem{Volt}
V.~E. Vol'kovi\v{c}. 
\newblock Numbers that are relatively prime to their number of prime divisors.
\newblock {\em Izv. Akad. Nauk USSR Ser. Fiz.-Math. Nauk}, 86(4):3--7, 1976.

\bibitem{Xuan}
T.~Z. Xuan.
\newblock On a problem of Erd\H{o}s and Ivi\'{c}.
\newblock {\em Publ. Inst. Math. (Beograd) (N.S.)}, 43(57):9--15, 1988.

\bibitem{Xuan2}
T.~Z. Xuan.
\newblock On a problem of Erd\H{o}s and Ivi\'{c}.
\newblock {\em J. Math. (Wuhan)}, 9(4):375--380, 1989.


\end{thebibliography}

\end{document}